\documentclass[12pt]{amsart}

\usepackage[utf8]{inputenc}
\usepackage[T1]{fontenc}

\usepackage{amsmath,amssymb,amsfonts,textcomp,amsthm,xifthen,graphicx,color,pgfplots}
\usepackage{ifthen}

\usepackage{fullpage}

\usepackage[utf8]{inputenc}

\usepackage[pdftex,
            pdfauthor={Thomas F\"uhrer and Dirk Praetorius},
            pdftitle={A short note on plain convergence of ALSFEM},
            ]{hyperref}

\newcommand{\HDivset}[1]{\underbar\HH(\Div;#1)}
\newcommand\Grad{\boldsymbol\nabla}
\def\Div{{\mathbf{div}\,}}

\newcommand{\supp}{\ensuremath{\operatorname{supp}}}

\newcommand{\MMM}{\ensuremath{\boldsymbol{M}}}
\newcommand{\JJ}{\ensuremath{\mathcal{J}}}
\newcommand{\HHH}{\ensuremath{\mathcal{H}}}

\newtheorem{theorem}{Theorem}
\newtheorem{lemma}[theorem]{Lemma}

\newtheorem{proposition}[theorem]{Proposition}
\newtheorem{remark}[theorem]{Remark}
\newtheorem{algorithm}[theorem]{Algorithm}

\newcommand{\refine}{\operatorname{refine}}

\def\enorm#1{|\hspace*{-.5mm}|\hspace*{-.5mm}|#1|\hspace*{-.5mm}|\hspace*{-.5mm}|}

\DeclareMathOperator{\V}{\mathbb{V}}

\newcommand{\ip}[2]{(#1\hspace*{.5mm},#2)}
\newcommand{\dual}[2]{\langle#1\hspace*{.5mm},#2\rangle}

\newcommand{\norm}[3][]{#1\|#2#1\|_{#3}}

\newcommand{\diam}{\mathrm{diam}}
\def\div{{\rm div\,}}

\newcommand\ccurl{\boldsymbol{\operatorname{curl}}\,}

\newcommand{\Hdivset}[1]{\boldsymbol{H}(\div;#1)}

\newcommand{\set}[2]{\big\{#1\,:\,#2\big\}}

\newcommand{\ran}{\operatorname{ran}}

\newcommand{\R}{\ensuremath{\mathbb{R}}}
\newcommand{\N}{\ensuremath{\mathbb{N}}}
\newcommand{\HH}{\ensuremath{{\boldsymbol{H}}}}

\newcommand{\LL}{\ensuremath{\mathcal{L}}}
\newcommand{\DD}{\ensuremath{\mathcal{D}}}

\newcommand{\nn}{\ensuremath{\mathbf{n}}}
\newcommand{\vv}{\ensuremath{\boldsymbol{v}}}
\newcommand{\ww}{\ensuremath{\boldsymbol{w}}}
\newcommand{\TT}{\ensuremath{\mathcal{T}}}
\newcommand{\TTT}{\ensuremath{\mathbb{T}}}
\newcommand{\MM}{\ensuremath{\mathcal{M}}}
\newcommand{\FF}{\ensuremath{\boldsymbol{F}}}
\newcommand{\RR}{\ensuremath{\boldsymbol{R}}}

\newcommand{\cS}{\ensuremath{\mathcal{S}}}

\newcommand{\PP}{\ensuremath{\mathcal{P}}}

\newcommand{\NN}{\ensuremath{\mathcal{N}}}

\newcommand{\normal}{\ensuremath{{\boldsymbol{n}}}}

\newcommand{\VV}{\ensuremath{\mathbf{V}}}

\newcommand{\ff}{\ensuremath{\boldsymbol{f}}}

\newcommand{\ssigma}{{\boldsymbol\sigma}}
\newcommand{\ttau}{{\boldsymbol\tau}}
\newcommand{\uu}{\boldsymbol{u}}

\def\fine{\circ}
\def\coarse{\bullet}

\makeatletter
\def\@seccntformat#1{\vspace*{-2mm}\newline\hspace*{4mm}%
  \protect\textup{\protect\@secnumfont
    \ifnum\pdfstrcmp{subsection}{#1}=0 \bfseries\fi%
    \csname the#1\endcsname
    \protect\@secnumpunct
  }%
}
\def\paragraph{\@startsection{paragraph}{4}%
  \z@\z@{-\fontdimen2\font}%
  {\normalfont\bfseries}}
\makeatother

\makeatletter
\def\section{\@startsection{section}{1}%
\z@{.7\linespacing\@plus\linespacing}{.5\linespacing}%
{\normalsize\scshape\bfseries\centering}}
\makeatother

\makeatletter
\renewcommand{\@secnumfont}{\bfseries}
\makeatother

\begin{document}

\title{A short note on plain convergence of\\ adaptive Least-squares finite element methods}
\date{\today}
\author{Thomas F\"uhrer}
\author{Dirk Praetorius}
\address{Facultad de Matem\'{a}ticas, Pontificia Universidad Cat\'{o}lica de Chile, Santiago, Chile}
\email{tofuhrer@mat.uc.cl \quad {\rm(corresponding author)}}
\address{TU Wien, Institute of Analysis and Scientific Computing, Wiedner Hauptstr.\ 8--10, 1040 Wien, Austria}
\email{dirk.praetorius@asc.tuwien.ac.at}

\thanks{{\bf Acknowledgment.} 
This work was supported by CONICYT (through FONDECYT project 11170050) and the Austrian Science Fund FWF  (through project P33216 and the special research program SFB F65).}

\begin{abstract}
We show that adaptive least-squares finite element methods driven by the canonical least-squares functional converge under weak conditions on PDE operator, mesh-refinement, and marking strategy. 
Contrary to prior works, our plain convergence does neither rely on sufficiently fine initial meshes nor on severe restrictions on marking parameters. Finally, we prove that convergence is still valid if a contractive iterative solver is used to obtain the approximate solutions (e.g., the preconditioned conjugate gradient method with optimal preconditioner).
The results apply within a fairly abstract framework which covers a variety of model problems.
\end{abstract}
\maketitle

\def\DD{\mathbb{D}}
\def\VV{\mathbb{V}}

\section{Introduction}

Least-squares finite element methods (LSFEMs) are a class of finite element methods that minimize the residual in some
norm. These methods often become rather easy to analyze and implement when the residual is measured in the space of square integrable functions. Some features of LSFEMs are the following: First, the resulting algebraic systems are always symmetric and positive definite, thus allowing the use of standard iterative solvers. 
Second, $\inf$--$\sup$ stability for any conforming finite element space is inherited from the continuous problem.
Finally, another feature is the built-in (localizable) error estimator that can be used to steer an adaptive mesh-refinement in,
e.g., an adaptive algorithm of the form
\begin{align*}
  \boxed{\texttt{SOLVE}} \quad\longrightarrow\quad
  \boxed{\texttt{ESTIMATE}} \quad\longrightarrow\quad
  \boxed{\texttt{MARK}} \quad\longrightarrow\quad
  \boxed{\texttt{REFINE}}.
\end{align*}

For standard discretizations, the mathematical understanding of adaptive finite element methods (AFEMs) has matured over the past decades. We mention~\cite{msv2008,Siebert11} for abstract theories for plain convergence of AFEM, the seminal works~\cite{doerfler96,mns2000,bdd2004,stevenson2007,ckns2008} for convergence of standard AFEM with optimal algebraic rates, and the abstract framework of~\cite{axioms}, which also provides a comprehensive review of the state of the art.

In contrast to this,
only very little is known about convergence of adaptive LSFEM; see~\cite{CCcollective20,BCS18,CP15,CPB17}.
To the best of our knowledge, plain convergence of adaptive LSFEM when using the built-in error estimator has so far only been addressed in~\cite{CPB17}. The other works~\cite{CCcollective20,BCS18,CP15} deal with optimal convergence results, 
but rely on alternative error estimators. 

The aim of this work is to shed some new light on the plain convergence of adaptive LSFEMs using the canonical error estimator shipped with them. 
Basically, the whole idea of this note is to verify that a large class of LSFEMs fits into the abstract framework of~\cite{Siebert11}.
From this, we then conclude that the sequence of discrete solutions produced by the above iteration converges to the exact solution of
the underlying problem (see Theorem~\ref{thm:convergence} below). 
Let us mention that our result is weaker than~\cite[Theorem~4.1]{CPB17} as we only show (plain) convergence whereas
\cite{CPB17} proves that an equivalent error quantity is contractive.
However, the latter result comes at the price of assuming sufficiently fine initial meshes and sufficiently large marking
parameters $0<\theta<1$ in the D\"orfler marking criterion. The latter contradicts somehow the by now standard proofs of optimal
convergence where $\theta$ needs to be sufficiently small; see~\cite{axioms} for an overview.
However, we also note that~\cite{CPB17} is constrained by the D\"orfler marking criterion, while the present analysis, in the spirit of~\cite{Siebert11}, covers a fairly wide range of marking strategies.

The remainder of the work is organized as follows: In Section~\ref{sec:main}, we state our assumptions on the PDE setting (Section~\ref{section:pde}), the mesh-refinement (Section~\ref{sec:mesh}), the discrete spaces (Section~\ref{sec:discreteSpaces}), and the marking strategy (Section~\ref{sec:marking}). Moreover, we recall the least-squares discretization (Section~\ref{section:least-squares}) as well as the built-in error estimator (Section~\ref{sec:estimator}) and formulate the common adaptive algorithm (Algorithm~\ref{algorithm}). Our first main result proves plain convergence of adaptive LSFEM (Theorem~\ref{thm:convergence}), if the LSFEM solution is computed exactly. In practice, however, iterative solvers (e.g.,  multigrid or the preconditioned CG method) are used. Our second main result (Theorem~\ref{thm:convergence:iexact}) proves convergence of adaptive LSFEM in the presence of inexact solvers (Algorithm~\ref{algorithm:inexact}). Overall, the presented abstract setting covers several model problems like the Poisson problem (Section~\ref{sec:poisson}), general elliptic second-order PDEs (Section~\ref{section:example:general}), linear elasticity (Section~\ref{sec:lame}), the magnetostatic Maxwell problem (Section~\ref{sec:maxwell}), and the Stokes problem (Section~\ref{sec:stokes}). While this work focusses on plain convergence, the short appendix (Appendix~\ref{section:optimal}) notes that LSFEM ensures discrete reliability so that, in the spirit of~\cite{axioms}, the only missing link for the mathematical proof of optimal convergence rates is the verification of linear convergence.

\def\Capx{C_{\rm apx}}
\def\Cloc{C_{\rm loc}}

\def\Cmon{c_{\rm cnt}}
\def\Ccnt{C_{\rm cnt}}
\def\exact{\star}
\def\LS{{\rm LS}}
\def\UU{\mathcal{U}}

\def\qctr{q_{\rm ctr}}
\def\nn{\underline{n}}

\def\tr{\mathrm{tr}}

\section{Plain convergence of adaptive least-squares methods}\label{sec:main}

\subsection{Continuous model formulation}\label{section:pde}

We consider a PDE in the abstract form
\begin{align}\label{eq:modelAbstract}
  \LL \uu^\exact = \FF \quad \text{in a bounded domain } \Omega \subset \R^d \text{ with } d \ge 2.
\end{align}
Here, $\FF\in L^2(\Omega)^N$ with $N \ge 1$ are given data, and $\LL \colon \VV(\Omega) \to L^2(\Omega)^N$ is a linear operator from some Hilbert space $\VV(\Omega)$ with norm $\norm{\cdot}{\VV(\Omega)}$ to $L^2(\Omega)^N$ with norm $\norm\cdot{L^2(\Omega)}$. For simplicity, we assume that homogeneous boundary conditions are contained in the space $\VV(\Omega)$. To abbreviate notation, we write $\norm{\cdot}{\omega} := \norm{\cdot}{L^2(\omega)}$ for any measurable set $\omega \subseteq \Omega$. Moreover, $\ip\cdot\cdot_\omega$ denotes the corresponding $L^2(\omega)$ scalar product. 

We make the following assumptions on $\LL$ and $\FF$:
\begin{enumerate}
\renewcommand{\theenumi}{A\arabic{enumi}}
\bf
\item\label{ass:pde}
\rm
\textbf{$\boldsymbol\LL$ is continuously invertible:} With constants $\Cmon, \Ccnt >0$, it holds that
    \begin{align*}
      \Cmon^{-1}\norm{\vv}{\VV(\Omega)}
      \le \norm{\LL \vv}{\Omega} 
      \le \Ccnt \norm{\vv}{\VV(\Omega)} \quad\text{for all } \vv\in \VV(\Omega).
    \end{align*}
\bf
\item\label{ass:solvability}
\rm
\textbf{PDE admits solution:} The given data satisfy $\FF\in \ran(\LL)$.
\end{enumerate}
While \eqref{ass:solvability} yields the existence of the solution $\uu^\exact \in \VV(\Omega)$ of~\eqref{eq:modelAbstract}, assumption~\eqref{ass:pde} leads to
\begin{align}\label{eq:central}
 \Cmon^{-1} \, \norm{\uu^\exact - \vv}{\VV(\Omega)} 
 \le \norm{\FF - \LL \vv }{\Omega}
 \le \Ccnt \, \norm{\uu^\exact - \vv}{\VV(\Omega)}
 \quad \text{for all } \vv \in \VV(\Omega)
\end{align}
and hence, in particular, also guarantees uniqueness.

In practice, the norm on $\VV(\Omega)$ is often an integer order Sobolev space (see the examples in Section~\ref{sec:examples})
that satisfies the following additional properties (which coincide with~\cite[eq.~(2.3)]{Siebert11}):
\begin{enumerate}
\renewcommand{\theenumi}{A\arabic{enumi}}
\setcounter{enumi}{2}
\bf
\item\label{ass:additivity}
\rm
\textbf{Additivity:} 
The norm on $\VV(\Omega)$ is additive, i.e., for two disjoint subdomains $\omega_1, \omega_2 \subset \Omega$ with positive Lebesgue measure, it holds that
\begin{align*}
  \norm{\vv}{\VV(\omega_1\cup\omega_2)}^2 
  = \norm{\vv}{\VV(\omega_1)}^2 + \norm{\vv}{\VV(\omega_2)}^2
  \quad \text{for all } \vv \in \VV(\Omega).
\end{align*}
\bf
\item\label{ass:abscont}
\rm
\textbf{Absolute continuity:}
The norm on $\VV(\Omega)$ is absolutely continuous with respect to the Lebesgue measure, i.e., for all $\vv \in \VV(\Omega)$ and all domains $\omega \subset \Omega$, it holds that
\begin{align*}
  \norm{\vv}{\VV(\omega)} \to 0 \quad\text{as}\quad |\omega|\to 0.
\end{align*}
\end{enumerate}

\subsection{Least-squares method}\label{section:least-squares}

Let $\VV_\coarse(\Omega)$ be a closed subspace of $\VV(\Omega)$. The least-squares method seeks $\uu_\coarse^\exact \in \VV_\coarse(\Omega)$ as the minimizer of the least-squares functional, i.e.,
\begin{align}\label{eq:minization}
  \LS(\uu_\coarse^\exact;\FF) = \min_{\vv_\coarse\in \VV_\coarse(\Omega)} \LS(\vv_\coarse;\FF),
  \quad \text{where} \quad
  \LS(\vv;\FF) = \norm{\FF - \LL\vv}{\Omega}^2.
\end{align}
The Euler--Lagrange equations for this problem read:
Find $\uu_\coarse^\exact\in \VV_\coarse(\Omega)$ such that 
\begin{align}\label{eq:varform}
  b(\uu_\coarse^\exact, \vv_\coarse) = F(\vv_\coarse) \quad\text{for all } \vv_\coarse \in \VV_\coarse(\Omega),
\end{align}
where
\begin{align*}
  b(\ww,\vv) := \ip{\LL\ww}{\LL\vv}_\Omega
  \quad \text{and} \quad 
  F(\vv):= \ip{\FF}{\LL\vv}_\Omega
  \quad\text{for all } \vv, \ww \in \VV(\Omega).
\end{align*}
It is straightforward to see that $F(\cdot)$ and $b(\cdot,\cdot)$ satisfy the assumptions of the Lax--Milgram lemma (and, in particular,~\cite[eq.~(2.1)--(2.2)]{Siebert11}), i.e., for all $\vv, \ww \in \VV(\Omega)$, it holds that 
\begin{align*}
  |F(\vv)| 
  \le \Ccnt \, \norm{\FF}{\Omega} \norm{\vv}{\VV(\Omega)}, 
  \quad
  |b(\ww,\vv)|
  \le \Ccnt^2 \, \norm{\ww}{\VV(\Omega)}\norm{\vv}{\VV(\Omega)},
 \quad
  \Cmon^{-2} \, \norm{\ww}{\VV(\Omega)}^2  
  \le b(\ww,\ww).
\end{align*}
Therefore, the discrete variational formulation~\eqref{eq:varform} admits a unique solution and is equivalent to the minimization problem~\eqref{eq:minization}. In particular, there holds the C\'ea lemma 
\begin{align}\label{eq:cea}
\begin{split}
  \Cmon^{-1}\norm{\uu^\exact - \uu_\coarse^\exact}{\V(\Omega)}
  \stackrel{\eqref{ass:pde}}\le \norm{\LL(\uu^\exact-\uu_\coarse^\exact)}\Omega
  &\stackrel{\eqref{eq:minization}}{=} \min_{\vv_\coarse \in \V_\coarse(\Omega)} \norm{\LL(\uu^\exact - \vv_\coarse)}{\Omega}
  \\ &\stackrel{\eqref{ass:pde}}\le \Ccnt \,\min_{\vv_\coarse \in \V_\coarse(\Omega)} \norm{\uu^\exact - \vv_\coarse}{\V(\Omega)},
\end{split}
\end{align}%
i.e., the exact least square solutions are quasi-optimal.
Moreover, Galerkin orthogonality
\begin{align}\label{eq:galerkinorthogonality}
  b(\uu^\exact-\uu_\coarse^\exact,\vv_\coarse) = 0 \quad\text{for all }\vv_\coarse \in \VV_\coarse(\Omega)
\end{align}
follows from the variational formulation~\eqref{eq:varform}.

\subsection{Meshes and mesh-refinement}\label{sec:mesh}

For simplicity, as in~\cite{Siebert11}, we restrict our presentation to conforming simplicial triangulations $\TT_\coarse$ of $\Omega$ and refinement by, e.g., the newest vertex bisection algorithm~\cite{stevenson:NVB,KPP13}.
For each triangulation $\TT_\coarse$, let $h_\coarse \in L^\infty(\Omega)$ denote the associated mesh-size function given by $h_\coarse|_T = h_T = |T|^{1/d}$. 

Let $\refine(\cdot)$ denote the refinement routine. 
We write $\TT_\fine = \refine(\TT_\coarse,\MM_\coarse)$ if $\TT_\fine$ is generated from $\TT_\coarse$ by refining (at least) all marked elements $\MM_\coarse \subseteq \TT_\coarse$. Moreover, $\TTT(\TT_\coarse)$ denotes the set of all meshes that can be generated by an arbitrary but finite number of refinements of $\TT_\coarse$. Throughout, let $\TT_0$ be a given initial mesh and $\TTT := \TTT(\TT_0)$. 

We make the following assumptions (which essentially coincide with~\cite[eq.~(2.4)]{Siebert11}):
\begin{enumerate}
\renewcommand{\theenumi}{R\arabic{enumi}}
\bf
\item\label{ass:refinement:reduction}
\rm
\textbf{Reduction on refined elements:} The mesh-size function is monotone and contractive with constant $0<q_\mathrm{ref}<1$ on refined elements, i.e.,
for all $\TT_\coarse \in \TTT$ and $\TT_\fine \in \TTT(\TT_\coarse)$, it holds that
\begin{align*}
  h_\fine \le h_\coarse \quad \text{a.e.\ in } \Omega
  \quad \text{and} \quad
  h_\fine|_T \leq q_\mathrm{ref} h_\coarse|_T \quad\text{for all } T \in \TT_\fine \setminus \TT_\coarse.
\end{align*}
\bf
\item\label{ass:refinement:quasiuniform}
\rm
\textbf{Uniform shape regularity:} There exists a constant $\kappa>0$ depending only on the initial mesh $\TT_0$ such that
\begin{align*}
  \sup_{T\in\TT_\coarse} \frac{\diam(T)^d}{|T|} \leq \kappa \quad\text{for all }\TT_\coarse\in\TTT.
\end{align*}
\bf
\item\label{ass:refinement:marked}
\rm
\textbf{Marked elements are refined:} It holds that
\begin{align*}
  \MM_\coarse \cap \refine(\TT_\coarse, \MM_\coarse) = \emptyset
  \quad \text{for all } \TT_\coarse \in \TTT \text{ and all } \MM_\coarse \subseteq \TT_\coarse.
\end{align*}
\end{enumerate}

\def\AA{\mathcal A}
\subsection{Discrete spaces}\label{sec:discreteSpaces}

We assume that each mesh $\TT_\coarse \in \TTT$ is associated with some discrete space $\VV_\coarse(\Omega)$ satisfying the following properties (which coincide with~\cite[eq.~(2.5)]{Siebert11}):
\begin{enumerate}
\renewcommand{\theenumi}{S\arabic{enumi}}
\bf
\item\label{ass:conformity}
\rm
\textbf{Conformity:} The spaces are conforming and finite dimensional, i.e., 
\begin{align*}
 \VV_\coarse(\Omega) \subset \VV(\Omega) \text{ and }
 \dim(\VV_\coarse(\Omega))<\infty \quad\text{for all } \TT_\coarse\in \TTT.
\end{align*}
\bf
\item\label{ass:nested}
\rm
\textbf{Nestedness:} Mesh-refinement guarantees nested discrete spaces, i.e., 
\begin{align*}
 \VV_\coarse(\Omega)\subseteq \VV_\fine(\Omega) \quad\text{for all } \TT_\fine \in \TTT(\TT_\coarse).
\end{align*}
\end{enumerate}
Moreover, we assume that there exists a dense subspace $\DD(\Omega) \subset \VV(\Omega)$ with additive norm $\norm\cdot{\DD(\Omega)}$ (see~\eqref{ass:additivity} with $\VV(\cdot)$ being replaced by $\DD(\cdot)$) satisfying an approximation property:
\begin{enumerate}
\renewcommand{\theenumi}{S\arabic{enumi}}
\setcounter{enumi}{2}
\bf
\item\label{ass:localApproximation}
\rm
\textbf{Local approximation property:} 
There exist constants $\Capx > 0$ and $s > 0$ such that, 
for all $\TT_\coarse \in \TTT$, there exists an approximation operator $\AA_\coarse \colon \DD(\Omega)\to \VV_\coarse(\Omega)$ such that 
    \begin{align*}
      \norm{\vv-\AA_\coarse \vv}{\VV(T)} \le \Capx \, h_T^s \norm{\vv}{\DD(T)} \quad\text{for all } \vv \in \DD(\Omega) \text{ and all } T \in \TT_\coarse.
    \end{align*}
\end{enumerate}

\subsection{Natural least-squares error estimator}\label{sec:estimator}

Let $\TT_\coarse \in \TTT$. For all $\uu_\coarse\in \VV_\coarse(\Omega)$ and all $\UU_\coarse \subset \TT_\coarse$, we consider the contributions of the least-squares functional 
\begin{align}\label{eq:eta}
 \eta_\coarse(\UU_\coarse, u_\coarse) 
 := \bigg( \sum_{T \in \UU_\coarse} \eta_\coarse(T, u_\coarse)^2 \bigg)^{1/2},
 \quad \text{where} \quad
 \eta_\coarse(T, u_\coarse) := \norm{\FF - \LL\uu_\coarse}{T}.
\end{align}
To abbreviate notation, let $\eta_\coarse(u_\coarse) := \eta_\coarse(\TT_\coarse, u_\coarse) = \norm{\FF - \LL\uu_\coarse}{\Omega}$. From~\eqref{eq:central}, we see that
\begin{align}\label{eq:reliable-efficient}
  \Cmon^{-1} \, \norm{\uu^\exact - \uu_\coarse}{\VV(\Omega)} 
  \le \eta_\coarse(u_\coarse) = \LS(\uu_\coarse;\FF) = \norm{\LL(\uu^\exact - \uu_\coarse)}{\Omega}
  \le \Ccnt \, \norm{\uu^\exact - \uu_\coarse}{\VV(\Omega)},
\end{align}
i.e., the least-squares functional provides a natural (localizable) error estimator which is reliable (upper bound $\norm{\uu^\exact - \uu_\coarse}{\VV(\Omega)} \lesssim \eta_\coarse(u_\coarse)$) and efficient (lower bound $\eta_\coarse(u_\coarse) \lesssim \norm{\uu^\exact - \uu_\coarse}{\VV(\Omega)}$) with respect to the error $\norm{\uu^\exact-\uu_\coarse}{\VV(\Omega)}$ of \emph{any approximation} $\uu^\exact \approx \uu_\coarse \in \VV_\coarse(\Omega)$.
In contrast to usual error estimators for standard Galerkin FEM, no data approximation term appears explicitly in the definition of the estimator nor is needed to show the lower bound. 

\subsection{Marking strategy}\label{sec:marking}

We recall the following assumption from~\cite[Section~2.2.4]{Siebert11} on the marking strategy, where $\uu_\coarse \in \VV_\coarse(\Omega)$ is a fixed approximation $\uu_\coarse \approx \uu^\exact$:
\begin{enumerate}
\renewcommand{\theenumi}{M}
\bf
\item\label{ass:marking}
\rm
There exists a fixed function $g\colon \R_{\geq0}\to\R_{\geq0}$ being continuous at $0 = g(0)$ such that the set of
marked elements $\MM_\coarse \subseteq \TT_\coarse$ (corresponding to $\uu_\coarse$) satisfies that
\begin{align*}
  \max_{T \in\TT_\coarse \setminus \MM_\coarse} \eta_\coarse(T,\uu_\coarse)
  \le g \big( \max_{T\in\MM_\coarse} \eta_\coarse(T,\uu_\coarse) \big).
\end{align*}
\end{enumerate}

We note that the marking assumption~\eqref{ass:marking} is clearly satisfied with $g(s) = s$ if $\MM_\coarse$ contains at least one element with maximal error indicator, i.e., there exists $T \in \MM_\coarse$ such that $\eta_\coarse(T, u_\coarse) \ge \eta_\coarse(T', u_\coarse)$ for all $T' \in \TT_\coarse$.
We recall from~\cite[Section~4.1]{Siebert11} that the latter is always satisfied for the following common marking strategies (so that the adaptivity parameter $\theta$ could even vary between the steps of the adaptive algorithm):
\begin{itemize}
\item \textit{Maximum strategy:} Given $0<\theta\leq 1$, the marked elements are determined by
\begin{align*}
 \MM_\coarse = \set{T\in\TT_\coarse}{\eta_\coarse(T,\uu_\coarse)\geq \theta\, M_\coarse},
 \quad \text{where} \quad 
 M_\coarse := \max_{T\in\TT_\coarse} \eta_\coarse(T,\uu_\coarse).
\end{align*}
\item \textit{Equilibration strategy:} Given $0 < \theta \le 1$, the marked elements are determined by
\begin{align*}
 \MM_\coarse = \set{T\in\TT_\coarse}{\eta_\coarse(T,\uu_\coarse)^2\geq \theta\, \eta_\coarse(\uu_\coarse)^2/\#\TT_\coarse}.
\end{align*}
\item \textit{D\"orfler marking strategy:} Given $0<\theta\leq 1$, the set $\MM_\coarse \subseteq \TT_\coarse$ is chosen such that
\begin{align*}
 \theta \eta_\coarse(\uu_\coarse)^2 \leq \sum_{T\in\MM_\coarse} \eta_\coarse(T,\uu_\coarse)^2
 \quad \text{and} \quad
 \max_{T\in\TT_\coarse \setminus \MM_\coarse}\eta_\coarse(T,\uu_\coarse)
 \le \min_{T\in\MM_\coarse} \eta_\coarse(T,\uu_\coarse).
\end{align*}
\end{itemize}
We note that the second condition on the D\"orfler marking is not explicitly specified in~\cite{doerfler96}, but usually satisfied if the implementation is based on sorting~\cite{pp2020}.

\subsection{Convergent adaptive algorithm}

Our first theorem states convergence of the following basic algorithm in the sense that error as well as error estimator are driven to zero.

\begin{algorithm}\label{algorithm}
  {\bfseries Input:} Initial triangulation $\TT_0$.\\
  {\bfseries Loop:} For all $\ell = 0, 1, 2, \dots$, iterate the following steps~{\rm(i)--(iv)}:
\begin{itemize}
  \item[\rm(i)] \textbf{SOLVE.} Compute the exact least-squares solution $\uu_\ell^\exact \in \VV_\ell(\Omega)$ by solving~\eqref{eq:minization}--\eqref{eq:varform}.
  \item[\rm(ii)] \textbf{ESTIMATE.} For all $T \in \TT_\ell$, compute the contributions $\eta_\ell(T, \uu_\ell^\exact)$ from~\eqref{eq:eta}.
\item[\rm(iii)] \textbf{MARK.} Determine a set $\mathcal{M}_\ell \subseteq \mathcal{T}_\ell$ of marked elements satisfying~\eqref{ass:marking} for $\uu_\ell = \uu_\ell^\exact$.
\item[\rm(iv)] \textbf{REFINE.} Generate a new mesh $\mathcal{T}_{\ell+1} := {\rm refine}(\mathcal{T}_\ell,\mathcal{M}_\ell)$.
\end{itemize}
{\bfseries Output:} Sequences of approximations $\uu_\ell^\exact$ and corresponding error estimators $\eta_\ell(\uu_\ell^\exact)$.\qed
\end{algorithm}

The following theorem is our first main result.

\begin{theorem}\label{thm:convergence}
Suppose the assumptions~\eqref{ass:pde}--\eqref{ass:abscont},
\eqref{ass:refinement:reduction}--\eqref{ass:refinement:marked},
\eqref{ass:conformity}--\eqref{ass:localApproximation}, 
and~\eqref{ass:marking}.
In addition, we make the following assumption on the locality of the operator $\LL$:
\begin{enumerate}
\renewcommand{\theenumi}{L}
\bf
\item\label{ass:localOperator}
\rm
\textbf{Local boundedness:} \it There exists $\Cloc > 0$ such that, for all $\TT_\coarse \in \TTT$, it holds that
\begin{align*}
  \norm{\LL\vv}T \le \Cloc \norm{\vv}{\VV(\Omega_\coarse(T))} \quad\text{for all }\vv\in \VV(\Omega)
  \text{ and all } T \in \TT_\coarse,
\end{align*}
where $\Omega_\coarse(T) := \bigcup\set{T' \in \TT_\coarse}{T \cap T' \neq \emptyset} \subset \Omega$ denotes the patch of $T$.
\end{enumerate}
Then, Algorithm~\ref{algorithm} generates a sequence of approximations $(\uu_\ell^\exact)_{\ell\in\N_0}$ such that
\begin{align}
 \norm{\uu^\exact - \uu_\ell^\exact}{\VV(\Omega)} + \eta_\ell(\uu_\ell^\exact) \to 0
 \quad \text{as} \quad 
 \ell\to \infty.
\end{align}
\end{theorem}

\def\residual{\boldsymbol{R}}
\begin{proof}
We only need to check that the assumptions of~\cite[Theorem~2.1]{Siebert11} are satisfied. Our assumptions on marking strategy, refinement, and discrete spaces are the same as in~\cite[Section~2]{Siebert11}. The bilinear form of the least-squares FEM is coercive on $\VV(\Omega)$ and thus satisfies the uniform $\inf$--$\sup$ conditions from~\cite[eq.~(2.6)]{Siebert11}. It thus only remains to verify~\cite[eq.~(2.10)]{Siebert11}: Given an element $\ww\in \VV(\Omega)$, define the residual $\residual(\ww) \in \VV(\Omega)'$ by
\begin{align}\label{eq:def:residual}
 \dual{\residual(\ww)}{\vv} := F(\vv) - b(\ww,\vv) = b(\uu^\exact - \ww,\vv) 
 \quad\text{for all } \vv\in \VV(\Omega).
\end{align}
Let $\TT_\coarse \in \TTT$ and $\vv \in \VV(\Omega)$.
By definition and~\eqref{ass:localOperator}, we have that
\begin{align*}
 \dual{\residual(\uu_\coarse^\exact)}\vv 
 = \ip{\FF-\LL\uu_\coarse^\exact}{\LL\vv}_\Omega
 = \sum_{T\in\TT} \ip{\FF-\LL\uu_\coarse^\exact}{\LL\vv}_T
 \le \Cloc \sum_{T\in\TT} \eta_\coarse(T, \uu_\coarse^\exact) \, \norm{\vv}{\VV(\Omega_\coarse(T))},
\end{align*}
which is~\cite[eq.~(2.10a)]{Siebert11}. Moreover, it holds that
\begin{align*}
 \eta_\coarse(T, \uu_\coarse^\exact) \le \norm{\FF}{T} + \Cloc \, \norm{\uu_\coarse^\exact}{\VV(\Omega_\coarse(T))}
 \quad \text{for all } T \in \TT_\coarse.
\end{align*}
which is~\cite[eq.~(2.10b)]{Siebert11} and hence concludes the verification of~\cite[eq.~(2.10)]{Siebert11}. 
Altogether,~\cite[Theorem~2.1]{Siebert11} applies and proves that $\norm{\uu^\exact - \uu_\ell^\exact}{\VV(\Omega)} \to 0$ as $\ell \to \infty$. Since $\eta_\ell(\uu_\ell^\exact) \simeq \norm{\uu^\exact - \uu_\ell^\exact}{\VV(\Omega)}$, we also have that $\eta_\ell(\uu_\ell^\exact) \to 0$ as $\ell\to \infty$.
\end{proof}

\begin{remark}\label{rem:lowerBound}
Note that the locality assumption~\eqref{ass:localOperator} also directly implies local efficiency
\begin{align*}
 \eta_\coarse(T,u_\coarse) 
 = \norm{\LL(\uu^\exact - \uu_\coarse)}T 
 \le \Cloc \, \norm{\uu^\exact-\uu_\coarse}{\VV(\Omega_\coarse(T))}
 \text{ for all } T \in \TT_\coarse 
 \text{ and all } \uu_\coarse \in \VV_\coarse(\Omega).
\end{align*}
Again, we stress that this local lower bound does not include data approximation terms.
\end{remark}

\subsection{Convergence in the case of inexact solvers}\label{sec:inexact}

For given $\TT_\coarse \in \TTT$, the exact computation of the least-squares solution $\uu_\bullet^\exact \in \VV_\bullet(\Omega)$ corresponds to the solution of a symmetric and positive definite algebraic system; see~\eqref{eq:minization}--\eqref{eq:varform}. Let us assume that we have a contractive iterative solver at hand:
\begin{enumerate}
\renewcommand{\theenumi}{C}
\bf
\item\label{ass:contraction}
\rm\textbf{Contractive iterative solver:}
There exists $0 < \qctr < 1$ as well as an equivalent norm $\enorm\cdot$ on $\VV(\Omega)$ such that, for all $\TT_\coarse \in \TTT$, there exists $\Phi_\coarse \colon \VV_\bullet(\Omega) \to \VV_\bullet(\Omega)$ with
\begin{align*}
 \enorm{\uu_\coarse^\exact - \Phi_\coarse(\vv_\coarse)}
 \le \qctr \, \enorm{\uu_\coarse^\exact - \vv_\coarse}
 \quad \text{for all } \vv_\coarse \in \VV_\coarse(\Omega).
\end{align*}
\end{enumerate}
Under this additional assumption, the following adaptive strategy steers adaptive mesh-refinement as well as the iterative solver, where we employ nested iteration in Algorithm~\ref{algorithm:inexact}(iv) to lower the number of solver steps.

\begin{algorithm}\label{algorithm:inexact}
  {\bfseries Input:} Initial triangulation $\TT_0$, initial guess $\uu_{0,0} \in \VV_0(\Omega)$\\
  {\bfseries Loop:} For all $\ell = 0, 1, 2, \dots$, iterate the following steps~{\rm(i)--(iv)}:
\begin{itemize}
  \item[\rm(i)] \textbf{INEXACT SOLVE.} Starting from $\uu_{\ell,0}$, do at least one step of the iterative solver 
  \begin{align*}
   \uu_{\ell,n} := \Phi_\ell(\uu_{\ell,n-1}) \quad \text{for all } n = 1, \dots, \nn = \nn(\ell) \ge 1.
  \end{align*} 
  \item[\rm(ii)] \textbf{ESTIMATE.} For all $T \in \TT_\ell$, compute the contributions $\eta_\ell(T, \uu_{\ell, \nn})$ from~\eqref{eq:eta}.
\item[\rm(iii)] \textbf{MARK.} Determine a set $\mathcal{M}_\ell \subseteq \mathcal{T}_\ell$ of marked elements satisfying~\eqref{ass:marking} for $\uu_\ell = \uu_{\ell, \nn}$.
\item[\rm(iv)] \textbf{REFINE.} Generate a new mesh $\mathcal{T}_{\ell+1} := {\rm refine}(\mathcal{T}_\ell,\mathcal{M}_\ell)$ and define $\uu_{\ell+1,0} := \uu_{\ell,\nn}$.
\end{itemize}
{\bfseries Output:} Sequences of approximations $\uu_{\ell, \nn}$ and corresponding error estimators $\eta_\ell(\uu_{\ell, \nn})$.\qed
\end{algorithm}

\begin{remark}
For the plain convergence result of the following theorem, one solver step (i.e., $\nn(\ell) = 1$ for all $\ell \in \N_0$) is indeed sufficient. In practice, the steps \textbf{INEXACT SOLVE} and \textbf{ESTIMATE} in Algorithm~\ref{algorithm:inexact}{\rm(i)--(ii)} are usually combined. 
With some parameter $\lambda > 0$, a natural stopping criterion for the iterative solver reads
\begin{align*}
 \enorm{\uu_{\ell,\nn} - \uu_{\ell,\nn-1}} \le \lambda \, \eta_\ell(\uu_{\ell-1,\nn})
 \quad \text{resp.} \quad
 \enorm{\uu_{\ell,\nn} - \uu_{\ell,\nn-1}} \le \lambda \, \eta_\ell(\uu_{\ell,\nn}).
\end{align*}
For adaptive FEM based on locally weighted estimators, we refer, e.g., to~\cite{agl2013,fembem:uzawa} for linear convergence of such a strategy (based on the first criterion) and to~\cite{ABEMsolve,banach} for linear convergence with optimal rates (based on the second criterion for sufficiently small $0 <  \lambda \ll 1$). Moreover, we note that the second criterion even allows to prove convergence with optimal rates with respect to the computational costs~\cite{banach2}.
\end{remark}

Our second theorem states convergence of Algorithm~\ref{algorithm:inexact} in the sense that, also in the case of inexact solution of the least-squares systems, the error as well as the error estimator are driven to zero.

\begin{theorem}\label{thm:convergence:iexact}
Suppose the assumptions~\eqref{ass:pde}--\eqref{ass:abscont},
\eqref{ass:refinement:reduction}--\eqref{ass:refinement:marked},
\eqref{ass:conformity}--\eqref{ass:localApproximation}, \eqref{ass:marking},
\eqref{ass:localOperator}, and~\eqref{ass:contraction}.
Then, Algorithm~\ref{algorithm:inexact} generates a sequence of approximations $(\uu_{\ell,\nn})_{\ell\in\N_0}$ such that
\begin{align}
 \norm{\uu^\exact - \uu_{\ell,\nn}}{\VV(\Omega)} + \eta_\ell(\uu_{\ell,\nn}) \to 0
 \quad \text{as} \quad 
 \ell\to \infty.
\end{align}
\end{theorem}

The proof of Theorem~\ref{thm:convergence:iexact} requires some preparations. First, we recall~\cite[Lemma~3.1]{Siebert11}, which already dates back to the seminal work~\cite{bv1984}.

\begin{lemma}\label{lem:aprioriconvergence}
Suppose assumptions~\eqref{ass:pde}--\eqref{ass:solvability} and~\eqref{ass:conformity}--\eqref{ass:nested}.
Let $(\uu_\ell^\exact)_{\ell\in\N_0}$ denote the sequence of exact least-squares solutions obtained by solving~\eqref{eq:minization}--\eqref{eq:varform} for the meshes generated by Algorithm~\ref{algorithm:inexact}.
Then,
\begin{align}
 \norm{\uu_\infty^\exact - \uu_\ell^\exact}{\VV(\Omega)}
 \to 0 \quad \text{as } \ell \to \infty,
\end{align}
where $\uu_\infty^\exact \in \VV_\infty(\Omega) := \overline{\bigcup_{\ell\in\N_0} \VV_\ell(\Omega)} \subseteq \VV(\Omega)$ solves (and is, in fact, the unique solution of)
\begin{align*}
 b(\uu_\infty^\exact, \vv_\infty) = F(\vv_\infty) 
 \quad \text{for all } \vv_\infty \in \VV_\infty(\Omega).
 \qquad \qed
\end{align*}
\end{lemma}

The next lemma shows that the inexact least-squares solutions converge indeed towards the same limit as the exact least-squares solutions.

\begin{lemma}\label{lem:convergence}
Besides the assumptions from Lemma~\ref{lem:aprioriconvergence} suppose assumption~\eqref{ass:contraction}. 
Let $(\uu_{\ell,\nn})_{\ell\in\N_0}$ denote the (final) approximations computed in Algorithm~\ref{algorithm:inexact}.
Then,
\begin{align}
 \norm{\uu_\infty^\exact - \uu_{\ell,\nn}}{\VV(\Omega)}
 \to 0 \quad \text{as } \ell \to \infty,
\end{align}
where $\uu_\infty^\exact \in \VV(\Omega)$ is the limit from Lemma~\ref{lem:aprioriconvergence}.
\end{lemma}

\begin{proof}
Recall the norm $\enorm\cdot$ from~\eqref{ass:contraction}. Since $\nn(\ell+1) \ge 1$ and $\uu_{\ell+1,0} = \uu_{\ell,\nn}$, it holds that
\begin{align*}
 \enorm{\uu_{\ell+1}^\exact - \uu_{\ell+1,\nn}}
 \le \qctr^{\nn(\ell+1)} \, \enorm{\uu_{\ell+1}^\exact - \uu_{\ell+1,0}}
 \le \qctr \, \enorm{\uu_\ell^\exact - \uu_{\ell,\nn}}
 + \qctr \, \enorm{\uu_{\ell+1}^\exact - \uu_\ell^\exact}.
\end{align*}
Let $\alpha_\ell := \enorm{\uu_\ell^\exact - \uu_{\ell,\nn}} \ge 0$. Then, the latter estimate takes the form
\begin{align*} 
 0 \le \alpha_{\ell+1} \le \qctr \, \alpha_\ell + \enorm{\uu_{\ell+1}^\exact - \uu_\ell^\exact},
 \quad \text{where} \quad \lim_{\ell\to\infty} \enorm{\uu_{\ell+1}^\exact - \uu_\ell^\exact} = 0.
\end{align*}
Elementary calculus (see, e.g., the estimator reduction in~\cite[Corollary~4.8]{axioms}) shows that
\begin{align*}
 \norm{\uu_\ell^\exact - \uu_{\ell,\nn}}{\VV(\Omega)}
 \simeq \enorm{\uu_\ell^\exact - \uu_{\ell,\nn}} = \alpha_\ell \to 0
 \quad \text{as } \ell \to \infty.
\end{align*}
Overall, we thus see that
\begin{align*}
 \norm{\uu_\infty^\exact - \uu_{\ell,\nn}}{\VV(\Omega)}
 \le \norm{\uu_\infty^\exact - \uu_\ell^\exact}{\VV(\Omega)}
 + \norm{\uu_\ell^\exact - \uu_{\ell,\nn}}{\VV(\Omega)}
 \to 0 \quad \text{as } \ell \to \infty.
\end{align*}
This concludes the proof.
\end{proof}

\begin{proof}[Proof of Theorem~\ref{thm:convergence:iexact}]
We verify the validity of the building blocks of the proof of~\cite[Theorem~2.1]{Siebert11}  in the case of inexact solutions. This basically follows from the convergence $\enorm{\uu_\infty^\exact - \uu_{\ell,\nn}}\to 0$ shown in Lemma~\ref{lem:convergence}. Indeed, a closer look unveils that we only need to verify~\cite[Lemma~3.5--3.6 and Proposition~3.7]{Siebert11} which we do in the following steps:
  
\textbf{Step 1 (Uniform boundedness):}
From~\eqref{ass:pde}, \eqref{ass:additivity}, and Lemma~\ref{lem:convergence}, it follows that
\begin{align*}
 \eta_\ell(\uu_{\ell,\nn})
 = \norm{\FF - \LL \uu_{\ell,\nn}}{\Omega}
 \le \norm{\FF - \LL \uu_\infty^\exact}{\Omega}
 + \Ccnt \, \norm{\uu_\infty^\exact - \uu_{\ell,\nn}}{\VV(\Omega)}
\end{align*}
and hence $\displaystyle  \sup_{\ell \in \N_0} \eta_\ell(\uu_{\ell,\nn}) < \infty$. 

\textbf{Step 2 (Convergence of estimator on marked elements):}
With~\eqref{ass:pde} and~\eqref{ass:additivity}, it holds that
\begin{align*}
 \eta_\ell(T,\uu_{\ell,\nn})
 = \norm{\FF - \LL \uu_{\ell,\nn}}{T}
 \le \norm{\FF - \LL \uu_\infty^\exact}{T}
 + \Ccnt \, \norm{\uu_\infty^\exact - \uu_{\ell,\nn}}{\VV(\Omega)}.
\end{align*}
By Lemma~\ref{lem:convergence}, the second term tends to zero. Arguing as in the proof of~\cite[Lemma~3.6]{Siebert11} and hence exploiting~\eqref{ass:abscont} and~\eqref{ass:refinement:reduction}--\eqref{ass:refinement:marked}, we see that
\begin{align*}
 \lim_{\ell\to\infty} \max\set{\eta_\ell(T, \uu_{\ell,\nn})}{T\in\MM_\ell} = 0.
\end{align*}

\textbf{Step 3 (Weak convergence of residual):}
We tweak Step~2 in the proof of~\cite[Proposition~3.7]{Siebert11}: Let $\vv \in \DD(\Omega)$ and recall that $\DD(\Omega) \subset \VV(\Omega)$ is dense by assumption~\eqref{ass:localApproximation}. Recall the residual from~\eqref{eq:def:residual}. For the exact least-squares solution $\uu_\ell^\exact \in \VV_\ell(\Omega)$, we may use Galerkin orthogonality~\eqref{eq:galerkinorthogonality} and local boundedness~\eqref{ass:localOperator} to see that
\begin{align*}
 |\dual{\residual(\uu_{\ell,\nn})}{\vv}| 
 &\,=\, | b(\uu^\exact - \uu_{\ell,\nn}, \vv - \AA_\ell\vv )
 + b(\uu^\exact - \uu_{\ell,\nn}, \AA_\ell\vv ) |
 \\ &
 \,\stackrel{\eqref{eq:galerkinorthogonality}}{=}\, | b(\uu^\exact - \uu_{\ell,\nn}, \vv - \AA_\ell\vv )
 + b(\uu_\ell^\exact - \uu_{\ell,\nn}, \AA_\ell\vv ) | 
 \\
 &\!\stackrel{\eqref{ass:pde}}\lesssim\! |\ip{\LL(\uu^\exact - \uu_{\ell,\nn})}{\LL(\vv-\AA_\ell\vv)}_\Omega|
 + \norm{\uu_\ell^\exact - \uu_{\ell,\nn}}{\VV(\Omega)} \norm{\AA_\ell\vv}{\VV(\Omega)}
 \\&
 \stackrel{\eqref{ass:localOperator}}\lesssim \sum_{T\in\TT_\ell} \norm{\LL(\uu^\exact - \uu_{\ell,\nn})}T \norm{\vv-\AA_\ell\vv}{\VV(\Omega_\ell(T))}
 + \norm{\uu_\ell^\exact - \uu_{\ell,\nn}}{\VV(\Omega)} \norm{\AA_\ell\vv}{\VV(\Omega)}
 \\
 &\,=  \sum_{T\in\TT_\ell} \eta_\ell(T, \uu_{\ell,\nn}) \norm{\vv - \AA_\ell\vv}{\VV(\Omega_\ell(T))}
 + \norm{\uu_\ell^\exact - \uu_{\ell,\nn}}{\VV(\Omega)} \norm{\AA_\ell\vv}{\VV(\Omega)}.
\end{align*}
For fixed $\vv \in \DD(\Omega)$, it holds that $\norm{\uu_\ell^\exact-\uu_{\ell,\nn}}{\VV(\Omega)}\norm{\AA_\ell\vv}{\VV(\Omega)}\to 0$ by~\eqref{ass:localApproximation} and Lemma~\ref{lem:convergence}. 
The sum on the right-hand side can be dealt with as in~\cite[Proposition~3.7]{Siebert11}, and we conclude that
\begin{align*}
 \lim_{\ell\to\infty} \dual{\residual(\uu_{\ell,\nn})}\vv = 0 \quad\text{for all }\vv\in \DD(\Omega).
\end{align*}

\textbf{Step 4 (Convergence of inexact solutions and estimators):}
With the auxiliary results established above, we can follow the proof of~\cite[Theorem~2.1]{Siebert11} step by step to deduce that $\norm{\uu^\exact - \uu_{\ell,\nn}}{\VV(\Omega)} \to 0$ as $\ell \to \infty$.
Finally, the equivalence $\norm{\uu^\exact - \uu_{\ell,\nn}}{\VV(\Omega)} \simeq \eta_\ell(\uu_{\ell,\nn})$ concludes the proof.
\end{proof}

\subsection{Preconditioned conjugate gradient method (PCG)}

We recall a well-known result which follows from~\cite[Theorem~11.3.3]{matcomp}. 
The presented form is, e.g., found in~\cite[Lemma~1]{ABEMsolve}.

\begin{lemma}\label{lemma:pcg}
Let $\boldsymbol{A}_\coarse, \boldsymbol{P}_\coarse \in \mathbb{R}^{N \times N}$ be symmetric and positive definite,
$\boldsymbol{b}_\coarse\in \mathbb{R}^N$, $\boldsymbol{x}_\coarse^\exact := \boldsymbol{A}_\coarse^{-1}
\boldsymbol{b}_\coarse$, and $\boldsymbol{x}_{\coarse, 0} \in \mathbb{R}^N$.
Suppose the $\ell_2$-condition number estimate
\begin{align}\label{eq:pcg:condNumber}
 {\rm cond}_2(\boldsymbol{P}_\coarse^{-1/2} \boldsymbol{A}_\coarse \boldsymbol{P}_\coarse^{-1/2}) \le C_{\rm pcg}.
\end{align}
Then, the iterates $\boldsymbol{x}_{\coarse,n}$ of the PCG algorithm satisfy the contraction property
\begin{align}\label{eq:pcg:reduction}
  \norm{\boldsymbol{x}_\coarse^\exact - \boldsymbol{x}_{\coarse,n+1}}{\boldsymbol{A}_\coarse} 
 \le \qctr \, \norm{\boldsymbol{x}_\coarse^\exact - \boldsymbol{x}_{\coarse,n}}{\boldsymbol{A}_\coarse}
 \quad\text{for all } n \in \mathbb{N}_0,
\end{align}
where $\qctr := (1-1/C_{\rm pcg})^{1/2} < 1$ and $\norm{\boldsymbol{y}_\coarse}{\boldsymbol{A}_\coarse}^2 := \boldsymbol{y}_\coarse\cdot\boldsymbol{A}_\coarse\boldsymbol{y}_\coarse$ \ for $\boldsymbol{y}_\coarse \in \R^N$.\qed
\end{lemma}

Let $\{\boldsymbol{\phi}_{\coarse,1},\dots,\boldsymbol{\phi}_{\coarse,N}\}$ denote a basis of $\VV_\coarse(\Omega)$. Denote with $\boldsymbol{A}_\coarse$ the Galerkin matrix and with $\boldsymbol{b}_\coarse$ the right-hand side of the least-squares method with respect to that basis, i.e., $\boldsymbol A_\coarse[j,k] = b(\boldsymbol{\phi}_{\coarse,k},\boldsymbol{\phi}_{\coarse,j})$ and $\boldsymbol{b}_\coarse[j] = F(\boldsymbol{\phi}_{\coarse,j})$.

There is a one-to-one relation between vectors $\boldsymbol y_\coarse\in\R^N$ and functions $\vv_\coarse\in \VV_\coarse(\Omega)$ given by $\vv_\coarse =
\sum_{j=1}^N \boldsymbol y_\coarse[j] \boldsymbol{\phi}_{\coarse,j}$. 
Let $\uu_{\coarse,n}\in \VV_\coarse(\Omega)$ denote the function corresponding to the iterate $\boldsymbol{x}_{\coarse,n}$.
We note that the least-squares solution $\uu_\coarse^\exact \in \VV_\coarse(\Omega)$ corresponds to the coefficient vector $\boldsymbol{x}_{\bullet}^\exact := \boldsymbol{A}_\coarse^{-1}\boldsymbol{b}_\coarse$.
With the elementary identity
\begin{align*}
\enorm{\uu_\coarse^\exact-\uu_{\coarse,n}}^2 := \norm{\LL(\uu_\coarse^\exact-\uu_{\coarse,n})}{}^2 = b(\uu_\coarse^\exact-\uu_{\coarse,n},\uu_\coarse^\exact-\uu_{\coarse,n}) =
    \norm{\boldsymbol{x}_\coarse^\exact - \boldsymbol{x}_{\coarse,n}}{\boldsymbol{A}_\coarse}^2,
\end{align*}
the contraction property~\eqref{eq:pcg:reduction} thus reads
\begin{align}
  \enorm{\uu_\coarse^\exact-\uu_{\coarse,(n+1)}} \leq \qctr \, 
  \enorm{\uu_\coarse^\exact-\uu_{\coarse,n}} \quad\text{for all } n\in\N_0.
\end{align}
We make the following assumption on the preconditioner.
\begin{enumerate}
\renewcommand{\theenumi}{P}
\bf
\item\label{ass:preconditioner}
\rm
\textbf{Optimal preconditioner:} For all $\TT_\coarse \in \TTT$, there exists a symmetric preconditioner $\boldsymbol{P}_\bullet$ of the Galerkin matrix $\boldsymbol{A}_\bullet$ such that the constant in~\eqref{eq:pcg:condNumber} depends only on the initial mesh $\TT_0$.
\end{enumerate}
Under this assumption, PCG fits into the abstract framework from Section~\ref{sec:inexact}.

\def\RT{\boldsymbol{RT}}
\def\cS{S}
\section{Examples}\label{sec:examples}

In this section, we consider some common model problems. Throughout, we assume that one of the marking strategies from Section~\ref{sec:marking} is used.
Moreover, we assume that $\TT_0$ is a conforming simplicial triangulation of some bounded Lipschitz domain $\Omega \subset \R^d$ with $d = 2, 3$ and NVB is used for mesh-refinement (see Section~\ref{sec:mesh}). In particular, there holds assumption~\eqref{ass:marking} as well as~\eqref{ass:refinement:reduction}--\eqref{ass:refinement:marked}.
To conclude convergence of ALSFEM, it therefore remains to show that the assumptions~\eqref{ass:pde}--\eqref{ass:abscont},~\eqref{ass:localOperator}, and~\eqref{ass:conformity}--\eqref{ass:localApproximation} are satisfied for the following examples.

\subsection{Poisson problem}\label{sec:poisson}

For given $f\in L^2(\Omega)$, consider the Poisson problem
\begin{subequations}\label{eq:poisson}
\begin{alignat}{2}
  -\Delta u &= f &\quad&\text{in }\Omega, \\
  u &= 0 &\quad&\text{on } \Gamma:=\partial\Omega.
\end{alignat}
\end{subequations}
With the substitution $\ssigma := \nabla u$, this is equivalently reformulated as a first-order system
\begin{subequations}\label{eq:fo}
\begin{alignat}{2}
  -\div \ssigma &= f &\quad&\text{in }\Omega, \\
  \nabla u - \ssigma &= 0 &\quad&\text{in }\Omega, \\
  u &= 0 &\quad&\text{on } \Gamma.
\end{alignat}
\end{subequations}
With the Hilbert space
\begin{align}\label{eq:poisson:space}
  \VV(\Omega) &:= H_0^1(\Omega) \times \Hdivset\Omega,
\end{align}
the first-order system~\eqref{eq:fo} can equivalently be recast in the abstract form~\eqref{eq:modelAbstract}, i.e.,
\begin{align}\label{eq:poisson:first-order}
 \LL \begin{pmatrix} u \\ \ssigma \end{pmatrix}
 := \begin{pmatrix} -\div\ssigma \\ \nabla u-\ssigma \end{pmatrix}
 = \begin{pmatrix} f \\ \boldsymbol{0} \end{pmatrix} =: \FF \in L^2(\Omega)^{d+1}.
\end{align}
It is well-known (see, e.g., the textbook~\cite{BochevGunzburgerLSQ}) that $\LL \colon \V(\Omega) \to L^2(\Omega)^{d+1}$ is an isomorphism, so that~\eqref{ass:pde}--\eqref{ass:solvability} are guaranteed. Clearly,~\eqref{ass:additivity}--\eqref{ass:abscont} are satisfied, since $\norm{\cdot}{\V(\Omega)}$ relies on the Lebesgue norm $\norm{\cdot}{\Omega}$.
Moreover, assumption~\eqref{ass:localOperator} follows from
\begin{align*}
  \norm{\LL(v,\ttau)}{T}^2 = \norm{\div\ttau}{T}^2 + \norm{\nabla v - \ttau}{T}^2
  \lesssim \norm{\nabla v}{T}^2 + \norm{v}{T}^2 + \norm{\div\ttau}T^2 + \norm{\ttau}{T}^2
  = \norm{(v,\ttau)}{{\VV(T)}}^2.
\end{align*}
A common FEM discretization of the energy space~\eqref{eq:poisson:space} involves the conforming subspace
\begin{align}\label{eq:poisson:discrete}
  \VV_\coarse(\Omega) := \cS_0^{k+1}(\TT_\coarse) \times \RT^k(\TT_\coarse) \subset \V(\Omega)
\end{align}
where $\cS_0^{k+1}(\TT_\coarse) \subseteq H^1_0(\Omega)$ is the usual Courant FEM space of order $k+1$ and $\RT^k(\TT_\coarse) \subset \Hdivset{\Omega}$ is the Raviart--Thomas FEM space of order
$k\in\N_0$; see, e.g.,~\cite{BoffiBrezziFortin}. Clearly, there hold~\eqref{ass:conformity}--\eqref{ass:nested}, since $\TT_\fine \in \refine(\TT_\coarse)$ yields that $\VV_\coarse(\Omega) \subseteq \V_\fine(\Omega) \subset \V(\Omega)$.

\def\II{\mathcal{I}}
\def\RR{\mathcal{R}}
It remains to show the local approximation property~\eqref{ass:localApproximation}. To that end, recall that $C_c^\infty(\overline\Omega)$ is dense in $H^1_0(\Omega)$ and $C^\infty(\overline\Omega)^d$ is dense in $\Hdivset\Omega$; see, e.g.,~\cite[Theorem~2.4]{GiraultRaviartBook}. Therefore, 
\begin{align*}
 C_c^\infty(\overline\Omega) \times C^\infty(\overline\Omega)^d 
 \subset \DD(\Omega) := \big[H^2(\Omega) \cap H^1_0(\Omega) \big] \times H^2(\Omega)^d 
 \subset \V(\Omega) \text{ is dense}.
\end{align*}
Let $\II_\coarse \colon H^2(\Omega) \to \cS^1(\TT_\coarse)$ be the nodal interpolation operator onto the first-order Courant space. Then, it holds that
\begin{align*}
 \norm{v - \II_\coarse v}{H^1(T)} \lesssim h_T \norm{v}{H^2(T)}
  \quad \text{for all } v \in H^2(\Omega) \cap H^1_0(\Omega) \text{ and all } T \in \TT_\coarse.
\end{align*}
Let $\RR_\coarse \colon H^1(\Omega) \to \RT^0(\TT_\coarse)$ be the Raviart--Thomas projector onto the lowest-order Raviart--Thomas space.
Then, it holds that
\begin{align*}
 \norm{\ttau - \RR_\coarse \ttau}{\Hdivset{T}} \lesssim h_T \norm{\ttau}{H^2(T)}
  \quad \text{for all } \ttau \in H^2(\Omega)^d \text{ and all } T \in \TT_\coarse.
\end{align*}
Overall, the approximation operator $\AA_\coarse \colon \DD(\Omega) \to \V_\coarse(\Omega)$ defined by $\AA_\coarse(v, \ttau) := (\II_\coarse v, \RR_\coarse \ttau)$  satisfies~\eqref{ass:localApproximation} with $s=1$.

\subsection{General second-order problem}
\label{section:example:general}
Given $f\in L^2(\Omega)$, we consider the general second-order linear elliptic PDE 
\begin{subequations}\label{eq:general}
\begin{alignat}{2}
  -\div(\boldsymbol{A}\nabla u) + \boldsymbol{b}\cdot\nabla u + c\, u &= f &\quad&\text{in }\Omega, \\
  u &= 0 &\quad&\text{on } \Gamma,
\end{alignat}
\end{subequations}
where $\boldsymbol{A}_{jk},\, \boldsymbol{b}_j, \, c \in L^\infty(\Omega)$, and $\boldsymbol{A}$ is symmetric and uniformly positive definite. It follows from the Fredholm alternative that existence and uniqueness of the weak solution $u \in H^1_0(\Omega)$ to~\eqref{eq:general} is equivalent to the well-posedness of the homogeneous problem, i.e.,
\begin{align*}
  a(u_0,v) := \dual{\boldsymbol{A}\nabla u_0}{\nabla v}_\Omega + \dual{\boldsymbol{b}\cdot\nabla u_0 + c\, u_0}{v}_\Omega 
\end{align*}
satisfies that
\begin{align}\label{eq:general:well-posedness}
 \forall u_0 \in H^1_0(\Omega): \Big( \big[ \forall v \in H^1_0(\Omega)\quad a(u_0,v)=0 \big] \quad \Longrightarrow \quad
 u_0 = 0 \Big);
\end{align}
see, e.g.,~\cite{brenner-scott}. Clearly, the general problem~\eqref{eq:general} does not only cover the Poisson problem from
Section~\ref{sec:poisson} (with $\boldsymbol{A}$ being the identity, $\boldsymbol{b}=0$, and $c = 0$, where~\eqref{eq:general:well-posedness} follows from the Poincar\'e inequality), but also the Helmholtz problem (with $\boldsymbol{A}$ being the identity, $\boldsymbol{b}=0$, and
$c=-\omega^2<0$, provided that $\omega^2$ is not an eigenvalue of the Dirichlet--Laplace eigenvalue problem).

The first-order reformulation of~\eqref{eq:general} reads
\begin{align}\label{eq:general:first-order}
  \LL\begin{pmatrix} u\\ \ssigma \end{pmatrix} :=
  \begin{pmatrix}
    -\div\ssigma + \boldsymbol{b}\cdot\nabla u + c\, u \\
    \boldsymbol{A}\nabla u - \ssigma
  \end{pmatrix}
  = 
  \begin{pmatrix}
    f\\ \boldsymbol{0}
  \end{pmatrix}
  =: \FF \in L^2(\Omega)^{d+1}.
\end{align}
With the Hilbert space $\VV(\Omega)$ from~\eqref{eq:poisson:space}
and provided that problem~\eqref{eq:general} admits a unique solution $u\in H^1(\Omega)$, we
conclude with~\cite[Theorem~3.1]{CLMMcC1994} that~\eqref{ass:pde}--\eqref{ass:solvability} are satisfied.
As in Section~\ref{sec:poisson}, the assumptions~\eqref{ass:additivity}--\eqref{ass:abscont} are clearly satisfied and~\eqref{ass:localOperator} follows as all coefficients of $\LL$ are bounded.

As before, a common FEM discretization involves the conforming subspace $\VV_\coarse(\Omega)$ from~\eqref{eq:poisson:discrete},
and~\eqref{ass:conformity}--\eqref{ass:localApproximation} follow as in Section~\ref{sec:poisson}. Overall, there holds the following convergence result:

\begin{proposition}\label{proposition:example:general}
Under the well-posedness assumption~\eqref{eq:general:well-posedness} and for any marking strategy satisfying~\eqref{ass:marking}, the adaptive least-squares formulation of~\eqref{eq:general:first-order} based on $\VV_\coarse(\Omega)$ from~\eqref{eq:poisson:discrete} guarantees plain convergence in the sense of Theorem~\ref{thm:convergence} and Theorem~\ref{thm:convergence:iexact}.\qed
\end{proposition}

\def\0{\boldsymbol{0}}
\def\HHH{\underline{\boldsymbol{H}\!}\hspace*{.05em}}
\def\identity{\boldsymbol{I}}

\subsection{Linear Elasticity}\label{sec:lame}
Given $\ff\in\boldsymbol{L}^2(\Omega):=L^2(\Omega)^d$ find the displacement $\uu\in \HH_0^1(\Omega) := H_0^1(\Omega)^d$ and the stress $\MMM\in\HDivset\Omega := \set{\MMM\in L^2(\Omega)^{d\times d}}{\Div\MMM \in \boldsymbol{L}^2(\Omega)}$ satisfying 
\begin{align*}
  -\Div\MMM &= \ff, \\
  \MMM - \mathbb{C}\boldsymbol{\epsilon}(\uu) &= \0.
\end{align*}
Here, $\Div(\cdot)$ denotes the divergence operator applied to each row and $\boldsymbol\epsilon(\cdot) = \tfrac12(\Grad(\cdot)+\Grad(\cdot)^\top)$ is the symmetric gradient, where $(\Grad \boldsymbol{v})_{ij} = \partial_j \boldsymbol{v}_i$ is the Jacobian.
Given the Lam\'e parameters $\lambda, \mu >0$, the positive definite elasticity tensor is given by
\begin{align*}
  \mathbb{C}\boldsymbol{N} = (\lambda\tr\boldsymbol{N})\identity + 2\mu \boldsymbol{N}
\end{align*}
with $\identity$ denoting the $d\times d$ identity matrix and $\tr\boldsymbol{N} = \sum_{j=1}^d \boldsymbol{N}_{jj}$ the trace operator.
Consider the Hilbert space
\begin{align}\label{eq:lame}
  \VV(\Omega) := \HH_0^1(\Omega) \times \HHH(\Div;\Omega)
\end{align}
equipped with the norm
\begin{align*}
  \norm{(\vv,\boldsymbol{N})}{\VV(\Omega)}^2 = \norm{\mathbb{C}^{-1/2}\boldsymbol{N}}{\Omega}^2 + \norm{\Div\boldsymbol{N}}{\Omega}^2 + 
  \norm{\mathbb{C}^{1/2}\boldsymbol\epsilon(\vv)}{\Omega}^2,
\end{align*}
which is equivalent to the canonical norm by Korn's inequality and the properties of the elasticity tensor.
Then, the first-order system can be put into the abstract form~\eqref{eq:modelAbstract}, i.e.,
\begin{align}\label{eq:lame:first-order}
  \LL\begin{pmatrix} \uu \\ \MMM \end{pmatrix} := 
  \begin{pmatrix} -\Div\MMM \\ \mathbb{C}^{-1/2}\MMM - \mathbb{C}^{1/2}\boldsymbol{\epsilon}(\uu) \end{pmatrix}
  = \begin{pmatrix} \ff \\ \0 \end{pmatrix} =: \FF \in L^2(\Omega)^{d+d^2},
\end{align}
where we identify $d\times d$ matrices with $d^2\times 1$ vectors. 
It is well known that the linear elasticity problem is well-posed so that~\eqref{ass:pde}--\eqref{ass:solvability} are satisfied; see, e.g.,~\cite[Theorem~2.1]{CaiKorsaweStarke2005}.
The validity of the assumptions~\eqref{ass:additivity}--\eqref{ass:abscont} and~\eqref{ass:localOperator} follows as in Section~\ref{sec:poisson}.
Consider the conforming discrete space
\begin{align}\label{eq:lame:spaces:discrete}
  \VV_\coarse(\Omega) = \cS_0^{k+1}(\TT_\coarse)^d \times \RT^k(\TT_\bullet)^d \subset \VV(\Omega).
\end{align}
Here, $\RT^k(\TT_\bullet)^d$ denotes the space of matrix-valued functions, where each row is an element of $\RT^k(\TT_\coarse)$.
As in Section~\ref{sec:poisson}, we conclude that the assumptions~\eqref{ass:conformity}--\eqref{ass:localApproximation} are satisfied.
Overall, we then have the following result:
\begin{proposition}
For any marking strategy satisfying~\eqref{ass:marking}, the adaptive least-squares formulation of~\eqref{eq:lame:first-order} based on $\VV_\coarse(\Omega)$ from~\eqref{eq:lame:spaces:discrete} guarantees plain convergence in the sense of Theorem~\ref{thm:convergence} and Theorem~\ref{thm:convergence:iexact}.\qed
\end{proposition}

\subsection{Maxwell problem}\label{sec:maxwell}
We consider the case $d=3$ only. Given $\ff\in \boldsymbol{L}^2(\Omega)$ and $c\in L^\infty(\Omega)$, find 
$\uu\in\boldsymbol{H}_0(\ccurl;\Omega)$ and $\ssigma\in \boldsymbol{H}(\ccurl;\Omega)$ satisfying
\begin{align*}
  \ccurl \ssigma + c\, \uu &= \ff, \\
  \ccurl \uu - \ssigma &= \0,
\end{align*}
where
\begin{align*}
  \boldsymbol{H}(\ccurl;\Omega) &:= \set{\vv\in \boldsymbol{L}^2(\Omega)}{\ccurl\vv\in\boldsymbol{L}^2(\Omega)},
  \\
  \boldsymbol{H}_0(\ccurl;\Omega) &:= \set{\vv\in \boldsymbol{H}(\ccurl;\Omega)}{\vv\times\normal|_{\partial\Omega} = 0}.
\end{align*}
With the Hilbert space
\begin{align}\label{eq:maxwell:spaces}
  \VV(\Omega) := \boldsymbol{H}_0(\ccurl;\Omega) \times \boldsymbol{H}(\ccurl;\Omega)
\end{align}
equipped with the norm
\begin{align*}
  \norm{(\vv,\ttau)}{\VV(\Omega)}^2 = \norm{\vv}{\Omega}^2 
  + \norm{\ccurl\vv}{\Omega}^2 + \norm{\ttau}{\Omega}^2 
  + \norm{\ccurl\ttau}{\Omega}^2,
\end{align*}
the first-order system can be written in the abstract form~\eqref{eq:modelAbstract}, i.e.,
\begin{align}\label{eq:maxwell:first-order}
 \LL \begin{pmatrix} \uu \\ \ssigma \end{pmatrix}
 := \begin{pmatrix} \ccurl\ssigma + c\uu\\ \ccurl\uu-\ssigma \end{pmatrix}
 = \begin{pmatrix} \ff \\ \0 \end{pmatrix} =: \FF \in L^2(\Omega)^{6}.
\end{align}
The assumptions~\eqref{ass:pde}--\eqref{ass:solvability} are satisfied if $\operatorname{ess}\inf_{x\in\Omega} c> 0$ or $c=-\omega^2< 0$ and $\omega^2$ is not an eigenvalue of the cavity problem. The argumentation is similar to the Poisson case. 
For the case $c=-\omega^2$, we also refer to~\cite[Section~2.4]{CarstensenStorn2018}.
The validity of the assumptions~\eqref{ass:additivity}--\eqref{ass:abscont} and~\eqref{ass:localOperator} follows as in Section~\ref{sec:poisson}.
Using the conforming discrete space
\begin{align}\label{eq:maxwell:spaces:discrete}
  \VV_\coarse(\Omega) = \boldsymbol{N}_0^k(\TT_\bullet) \times \boldsymbol{N}^k(\TT_\bullet) 
  := (\boldsymbol{N}^k(\TT_\bullet)\cap \boldsymbol{H}_0(\ccurl;\Omega)) \times \boldsymbol{N}^k(\TT_\bullet)\subset \VV(\Omega),
\end{align}
where $\boldsymbol{N}^k(\TT_\bullet)$ denotes the N\'ed\'elec space of order $k\in \N_0$, this proves that the assumptions~\eqref{ass:conformity}--\eqref{ass:nested} are satisfied.
It remains to verify the local approximation property~\eqref{ass:localApproximation}.
Recall from~\cite{MonkMaxwell} that $C_c^\infty(\overline\Omega)^3$ is dense in $\boldsymbol{H}_0(\ccurl;\Omega)$ and $C^\infty(\overline\Omega)^3$ is dense in $\boldsymbol{H}(\ccurl;\Omega)$. This shows that
\begin{align*}
  C_c^\infty(\overline\Omega)^3\times C^\infty(\overline\Omega)^3 
  \subset  \DD(\Omega):=(H^2(\Omega)\cap H_0^1(\Omega))^3\times H^2(\Omega)^3 \subset \VV(\Omega)
\end{align*}
is dense.
Let $\mathcal{N}_\bullet\colon H^1(\Omega)^3\to \boldsymbol{N}^k(\TT_\bullet)$ denote the edge interpolation operator which satisfies
\begin{align*}
  \norm{\vv-\mathcal{N}_\bullet\vv}{\boldsymbol{H}(\ccurl;T)} \lesssim h_T \norm{\vv}{H^2(T)} \quad\text{for all }
  \vv\in H^2(\Omega)^3 \text{ and all } T\in\TT_\bullet;
\end{align*}
see, e.g.,~\cite[Section~3.5--3.6]{hiptmair2002}.
Note that the projection $\mathcal{N}_\bullet$ by definition also preserves homogeneous boundary conditions.
Therefore, $\AA_\bullet\colon \DD(\Omega)\to\VV_\bullet(\Omega)$ defined by $\AA_\bullet(\vv,\ttau) = (\RR_\bullet \vv,\RR_\bullet\ttau)$ satisfies assumption~\eqref{ass:localApproximation}.
Overall, we have the following result:

\begin{proposition}
For any marking strategy satisfying~\eqref{ass:marking}, the adaptive least-squares formulation of~\eqref{eq:maxwell:first-order} based on $\VV_\coarse(\Omega)$ from~\eqref{eq:maxwell:spaces:discrete} guarantees plain convergence in the sense of Theorem~\ref{thm:convergence} and Theorem~\ref{thm:convergence:iexact}.\qed
\end{proposition}

\subsection{Stokes problem}\label{sec:stokes}
Given $\ff\in \boldsymbol{L}^2(\Omega)$, find the velocity $\uu \in \HH_0^1(\Omega)$ and the pressure $p\in
L^2(\Omega)$ satisfying
\begin{align*}
  -\boldsymbol{\Delta}\uu + \nabla p &= \ff, \\
  \div\uu &= 0, \\
  \int_\Omega p \,dx &= 0.
\end{align*}
There are many different reformulations that are suitable to define LSFEMs; see~\cite[Chapter~7]{BochevGunzburgerLSQ}.
Here, we consider a \emph{stress--velocity--pressure} formulation (see, e.g.~\cite{CaiLeeWang04}), where the stress is defined by the relation
\begin{align*}
  \MMM = \boldsymbol\epsilon(\uu) - p\identity.
\end{align*}

Let $\Pi_\Omega v = |\Omega|^{-1} \ip{v}1_\Omega$ denote the $L^2(\Omega)$-orthogonal projection on constants.
Consider the Hilbert space
\begin{align}\label{eq:stokes:spaces}
  \VV(\Omega) := \HH_0^1(\Omega) \times \HHH(\Div;\Omega) \times L^2(\Omega)
\end{align}
equipped with the norm
\begin{align*}
  \norm{(\vv,\boldsymbol{N},q)}{\VV(\Omega)}^2 = \norm{\vv}{\boldsymbol{H}^1(\Omega)}^2 + \norm{\boldsymbol{N}}\Omega^2 
  + \norm{\Div\boldsymbol{N}}{\Omega}^2 + \norm{q}{\Omega}^2 + \norm{\Pi_\Omega q}{\Omega}^2.
\end{align*}
Then, the first-order system can be recast in the abstract form~\eqref{eq:modelAbstract},
i.e.,
\begin{align}\label{eq:stokes:first-order}
  \LL
  \begin{pmatrix}
    \uu \\ \MMM \\ p
  \end{pmatrix} :=
  \begin{pmatrix}
    -\Div\MMM \\
    \MMM-\boldsymbol\epsilon(\uu) + p\identity \\ 
    \div\uu \\
    \Pi_\Omega p
  \end{pmatrix}
  = 
  \begin{pmatrix}
    \ff \\
    \0 \\
    0 \\
    0
  \end{pmatrix} =: \FF \in L^2(\Omega)^{d+d^2+2},
\end{align}
where we identify $d\times d$ matrices as before with $d^2\times 1 $ vectors. 

Following the proof of~\cite[Theorem~3.2]{CaiLeeWang04} we conclude that~\eqref{ass:pde}--\eqref{ass:solvability} are guaranteed (the only difference to the work~\cite{CaiLeeWang04} is that we include the pressure condition directly in the first-order formulation).
Clearly,~\eqref{ass:additivity}--\eqref{ass:abscont} follow as in Section~\ref{sec:poisson}. Moreover,
we have included the term $\norm{\Pi_\Omega q}\Omega$ in the definition of the norm $\norm\cdot{\VV(\Omega)}$ to ensure
that~\eqref{ass:localOperator} is satisfied, i.e.,
\begin{align*}
  \norm{\LL(\vv,\boldsymbol{N},q)}T^2 &= \norm{\Div\boldsymbol{N}}T^2 + \norm{\boldsymbol{N}-\boldsymbol\epsilon(\vv)+q}{T}^2 
  + \norm{\div\vv}T^2 + \norm{\Pi_\Omega q}T^2 \\
  &\lesssim \norm{\vv}{\boldsymbol{H}^1(T)}^2 + \norm{\boldsymbol{N}}T^2 + \norm{\Div\boldsymbol{N}}T^2 + \norm{q}T^2 + \norm{\Pi_\Omega q}T^2 = \norm{(\vv,\boldsymbol{N},q)}{\VV(T)}^2
\end{align*}
holds for all $(\vv,\boldsymbol{N},q)\in \VV(\Omega)$.

Consider the conforming subspace
\begin{align}\label{eq:stokes:spaces:discrete}
  \VV_\coarse(\Omega) = \cS_0^{k+1}(\TT_\coarse)^d \times \RT^k(\TT_\bullet)^d \times P^k(\TT_\bullet) \subset \VV(\Omega),
\end{align}
where $P^k(\TT_\bullet)$ denotes the space of elementwise polynomials of degree $\leq k\in\N_0$.
It follows that the assumptions~\eqref{ass:conformity}--\eqref{ass:nested} are satisfied.
It remains to prove the local approximation assumption~\eqref{ass:localApproximation}.
For the first two components $(\boldsymbol{v}, \boldsymbol{N}) \in \HH_0^1(\Omega) \times \HHH(\Div;\Omega)$ in the space $\VV(\Omega)$, we argue as in Section~\ref{sec:poisson}. For the last component $q\in L^2(\Omega)$, we note that $H^1(\Omega)$ is dense in $L^2(\Omega)$. Let $\mathcal{P}_\bullet\colon L^2(\Omega)\to P^k(\TT_\bullet)$ denote the $L^2(\Omega)$-orthogonal projection and observe that $\Pi_\Omega(q-\mathcal{P}_\bullet q) = 0$. 
Then,
\begin{align*}
  \norm{q-\mathcal{P}_\bullet q}{T} + \norm{\Pi_\Omega(q-\mathcal{P}_\bullet q)}T = \norm{q-\mathcal{P}_\bullet q}T \lesssim h_T \norm{\nabla q}{T} 
  \quad\text{for all } q\in H^1(\Omega) \text{ and all } T\in\TT_\bullet.
\end{align*}
Altogether, we choose $\DD(\Omega) = (H^2(\Omega)\cap H_0^1(\Omega)) \times H^2(\Omega)^{d\times d} \times H^1(\Omega)$ and $\mathcal{A}_\bullet\colon \DD(\Omega)\to \VV(\Omega)$ with
$\mathcal{A}_\bullet(\vv,\boldsymbol{N},q) = (\mathcal{I}_\bullet \vv,\mathcal{R}_\bullet\boldsymbol{N},\mathcal{P}_\bullet q)$, where $\mathcal{I}_\bullet$ and $\mathcal{R}_\bullet$ denote the operators from Section~\ref{sec:poisson} applied to each row of $\vv$ resp. $\boldsymbol{N}$. 
Therefore, we conclude that also~\eqref{ass:localApproximation} is satisfied.
Overall, we have the following result:

\begin{proposition}
For any marking strategy satisfying~\eqref{ass:marking}, the adaptive least-squares formulation of~\eqref{eq:stokes:first-order} based on $\VV_\coarse(\Omega)$ from~\eqref{eq:stokes:spaces:discrete} guarantees plain convergence in the sense of Theorem~\ref{thm:convergence} and Theorem~\ref{thm:convergence:iexact}.\qed
\end{proposition}

\bibliographystyle{abbrv}
\bibliography{literature}

\def\J{\boldsymbol{J}}
\def\NN{\mathcal{N}}
\def\RR{\mathcal{R}}

\appendix
\section{Comments on optimal convergence rates}\label{section:optimal}

In this appendix, we consider the LSFEM discretization of the general second-order problem from Section~\ref{section:example:general}. 
While Proposition~\ref{proposition:example:general} already proves \emph{plain convergence}
\begin{align*}
 \norm{\uu^\star - \uu_\ell^\star}{\V(\Omega)} \to 0
 \quad \text{as } \ell \to \infty,
\end{align*}
the present section proves that optimal convergence rates (in the sense of, e.g.,~\cite{axioms}) would already follow from 
the D\"orfler marking criterion (see Section~\ref{sec:marking}) and \emph{linear convergence}
\begin{align}\label{eq:linear_convergence}
 \norm{\uu^\star - \uu_{\ell+n}^\star}{\V(\Omega)}
 \le C_{\rm lin} \, q_{\rm lin}^n \, \norm{\uu^\star - \uu_\ell^\star}{\V(\Omega)}
 \quad \text{for all } \ell, n \in \N_0
\end{align}
with generic constants $C_{\rm lin} > 0$ and $0 < q_{\rm lin} < 1$.

\subsection{Scott--Zhang-type operators}
\label{section:scott-zhang}

For fixed $p \ge 1$, let $\JJ_\coarse \colon H^1(\Omega) \to S^p(\TT_\coarse)$ be the Scott--Zhang projector from~\cite{SZ_90}.
We recall that $\JJ_\coarse$ preserves discrete boundary data so that $\JJ_\coarse \colon H^1_0(\Omega) \to S^p_0(\TT_\coarse)$ is well-defined. Besides this, we will only exploit the following two properties which hold for all $v \in H^1(\Omega)$, $v_\coarse \in S^p(\TT_\coarse)$, and $T \in \TT_\coarse$:
\begin{itemize}
\item[\rm(i)] {\bf local projection property:} if $v|_{\Omega_\coarse(T)} = v_\coarse|_{\Omega_\coarse(T)}$, then $(\JJ_\coarse v)|_T = v_\coarse|_T$;
\item[\rm(ii)] {\bf global $\boldsymbol{H^1}$-stability:} $\norm{(1-\JJ_\coarse) v}{H^1(\Omega)} \le C(1 + \diam(\Omega)) \, \norm{v}{H^1(\Omega)}$,
where $C > 0$ depends only on the polynomial degree $p \ge 1$ and the shape regularity of $\TT_\coarse$.
\end{itemize}

The recent work~\cite{egsv2019} has constructed a Scott--Zhang-type projector $\PP_\coarse \colon \Hdivset{\Omega} \to \RT^{p-1}(\TT_\coarse)$. While~\cite[Section~3.4]{egsv2019} also includes Neumann boundary conditions $\ssigma\cdot \boldsymbol{n} = 0$ on $\Gamma_N \subseteq \partial\Omega$, we only consider the full space $\Hdivset{\Omega}$. Moreover, we will only exploit the following two properties which hold for all $\ssigma\in \Hdivset{\Omega}$, $\ssigma_\coarse \in \RT^{p-1}(\TT_\coarse)$, and $T \in \TT_\coarse$:
\begin{itemize}
\item[\rm(i)] {\bf local projection property:} if $\ssigma|_{\Omega_\coarse(T)} = \ssigma_\coarse|_{\Omega_\coarse(T)}$, then $(\PP_\coarse \ssigma)|_T = \ssigma_\coarse|_T$;
\item[\rm(ii)] {\bf global $\boldsymbol{H({\rm div};\Omega)}$-stability:} $\norm{(1-\PP_\coarse) \ssigma}{\Hdivset{\Omega}} \le C(1 + \diam(\Omega)) \, \norm{\ssigma}{\Hdivset{\Omega}}$, where again $C > 0$ depends only on $p \ge 1$ and the shape regularity of $\TT_\coarse$.
\end{itemize}

\subsection{Discrete reliability}

In the following, we exploit the local projection property and the global stability of the Scott--Zhang-type operators from Section~\ref{section:scott-zhang} and prove that the built-in least-squares estimator satisfies discrete reliability. 

\begin{lemma}[Discrete reliability]
There exist constants $C_\mathrm{ref},C_\mathrm{drel}\geq 1$ such that
for all $\TT_\coarse\in\refine(\TT_0)$ and all $\TT_\fine\in\refine(\TT_\coarse)$ there exists
$\RR_{\coarse,\fine}\subset\TT_\coarse$ such that
\begin{align}\label{eq:discrete_reliability}
 \norm{\uu_\fine^\star-\uu_\coarse^\star}{\VV(\Omega)} \leq C_\mathrm{drel} \eta_\coarse(\RR_{\coarse,\fine})
 \quad \text{and} \quad
 \#\RR_{\coarse,\fine} \leq C_\mathrm{ref}(\#\TT_\fine-\#\TT_\coarse).
\end{align}
The constant $C_\mathrm{ref}$ depends only on shape regularity of $\TT_\coarse$, whereas $C_\mathrm{drel}$ depends additionally on the constants $\Cmon$, $\Ccnt$ from~\eqref{ass:pde}, the polynomial degree $p \in \N$, and $\diam(\Omega)$.
\end{lemma}

\begin{proof}
The proof is split into three steps.

{\bf Step~1.}
Define the set
\begin{align*}
 \RR_{\coarse,\fine}=\TT_\coarse\setminus\NN_{\coarse,\fine} 
 \quad \text{with} \quad 
 \NN_{\coarse,\fine} = \set{T\in\TT_\bullet}{\Omega_\coarse(T)\subset \TT_\coarse\cap\TT_\fine}.
\end{align*}
Given $T \in \RR_{\coarse,\fine}$, there exists $T' \in \TT_\coarse$ such that $T \cap T' \neq \emptyset$ and $T' \not\in \TT_\coarse \cap \TT_\fine$. This implies that $T' \in \TT_\coarse \backslash \TT_\fine$ and hence $T$ belongs to the patch around $\TT_\coarse \backslash \TT_\fine$, i.e., $T \in \set{T' \in \TT_\coarse}{T' \cap \overline{\bigcup(\TT_\coarse \backslash \TT_\fine)} \neq \emptyset}$. Overall, we thus conclude that 
\begin{align*}
 \#\RR_{\coarse,\fine}
 \le \# \set{T' \in \TT_\coarse}{T' \cap \overline{\bigcup(\TT_\coarse \backslash \TT_\fine)} \neq \emptyset}
 \stackrel{\eqref{ass:refinement:quasiuniform}}\lesssim \#(\TT_\coarse \backslash \TT_\fine)
 \le \#\TT_\fine-\#\TT_\coarse.
\end{align*}%

{\bf Step~2.}
We define the operator $\II_\coarse \colon \V(\Omega) \to \V_\coarse(\Omega)$ by $\II_\coarse \vv := (\JJ_\coarse v, \PP_\coarse \ssigma)$ for $\vv = (v,\ssigma) \in \V(\Omega)$. Since $\JJ_\coarse$ and $\PP_\coarse$ are stable projections, it follows that also $\II_\coarse$ is a stable projection, i.e., for all $\vv \in \V(\Omega)$ and $\vv_\coarse \in \VV_\coarse(\Omega)$, it holds that
\begin{align}\label{eq:optimal1}
 \norm{(1-\II_\coarse)\vv}{\V(\Omega)}
 = \norm{(1 - \II_\coarse)(\vv \!-\!\vv_\coarse)}{\V(\Omega)}
 \lesssim \norm{\vv-\vv_\coarse}{\V(\Omega)},
\end{align}
where the hidden constant depends on $\diam(\Omega)$ and the constants $C > 0$ from Section~\ref{section:scott-zhang}.
Moreover, let $\uu_\fine^\star = (u_\fine^\star,\ssigma_\fine^\star)$. From the local projection properties of the operators $\JJ_\coarse$ and $\PP_\coarse$ and the choice of $\NN_{\coarse,\fine}$, it follows that $\uu_\fine^\star|_{\NN_{\coarse,\fine}} = (\II_\coarse\uu_\fine^\star)|_{\NN_{\coarse,\fine}}$. This leads to $\supp(u_\fine^\star - \II_\coarse u_\fine^\star) \subseteq \bigcup(\TT_\coarse \backslash \NN_{\coarse,\fine}) = \bigcup \RR_{\coarse,\fine}$. Since $\LL$ is defined pointwise, this also implies that
\begin{align}\label{eq:optimal2}
  \supp \LL(u_\fine^\star - \II_\coarse u_\fine^\star) \subseteq \mbox{$\bigcup$} \RR_{\coarse,\fine}.
\end{align}

{\bf Step~3.}
For any $\vv_\coarse\in \VV_\coarse(\Omega)$, the Galerkin orthogonality~\eqref{eq:galerkinorthogonality} proves that
\begin{align*}
 \norm{\uu_\fine^\star-\uu_\coarse^\star}{\VV(\Omega)}^2 
 &\stackrel{{\eqref{ass:pde}}}\simeq b(\uu_\fine^\star-\uu_\coarse^\star,\uu_\fine^\star-\uu_\coarse^\star) 
 \stackrel{\eqref{eq:galerkinorthogonality}}= b(\uu-\uu_\coarse^\star,\uu_\fine^\star-\uu_\coarse^\star)
 \stackrel{\eqref{eq:galerkinorthogonality}}= b(\uu-\uu_\coarse^\star,\uu_\fine^\star-\vv_\coarse).
\end{align*}
With the choice $\vv_\coarse = \II_\coarse u_\fine^\star$, we see that
\begin{align*}
 &b(\uu-\uu_\coarse^\star,\uu_\fine^\star-\vv_\coarse) 
 =
 \ip{\LL(\uu-\uu_\coarse^\star)}{\LL(\uu_\fine^\star-\vv_\coarse)}_\Omega 
 \stackrel{\eqref{eq:optimal2}}=
 \ip{\LL(\uu-\uu_\coarse^\star)}{\LL(\uu_\fine^\star-\vv_\coarse)}_{\bigcup\RR_{\coarse,\fine}}
 \\& \quad
 \stackrel{\phantom{\eqref{ass:pde}}}\leq
 \norm{\LL(\uu-\uu_\coarse^\star)}{\bigcup\RR_{\coarse,\fine}} \, 
 \norm{\LL(\uu_\fine^\star-\vv_\coarse)}{\bigcup\RR_{\coarse,\fine}} 
 \stackrel{\eqref{eq:optimal2}}= 
 \eta_\coarse(\RR_{\coarse,\fine}) \, \norm{\LL(\uu_\fine^\star-\vv_\coarse)}{\Omega}
 \\& \quad
 \stackrel{\eqref{ass:pde}}\lesssim \eta_\coarse(\RR_{\coarse,\fine})  \norm{\uu_\fine^\star-\vv_\coarse}{\VV(\Omega)}
 \stackrel{\eqref{eq:optimal1}}\lesssim \eta_\coarse(\RR_{\coarse,\fine})  \norm{\uu_\fine^\star-\uu_\coarse^\star}{\VV(\Omega)}.
\end{align*}
Combining the latter two estimates, we conclude the proof.
\end{proof}
\color{black}

\def\opt{\diamond}
\subsection{Linear convergence implies optimal algebraic convergence rates}

Given $s > 0$, we consider the following approximation class 
\begin{align}
 \norm{\uu}{\mathbb{A}_s} := \sup_{N \in \N_0} \min_{\substack{\TT_\opt \in \mathbb{T} \\ \#\TT_\opt - \#\TT_0 \le N}} \min_{\vv_\opt \in \V_\opt(\Omega)} (N+1)^s \norm{\uu^\star - \vv_\opt}{\V(\Omega)} \in \R_{\ge0} \cup \{\infty\}.
\end{align}
We note that $\norm{\uu}{\mathbb{A}_s} < \infty$ implies the existence of a sequence of meshes $(\bar\TT_\ell)_{\ell \in \N_0}$ with $\bar\TT_0 = \TT_0$ such that the corresponding least-square solutions satisfy
\begin{align*}
 \min_{\bar\vv_\ell \in \bar\V_\ell(\Omega)}\norm{\uu^\star - \bar\vv_\ell}{\V(\Omega)} 
 \lesssim (\#\bar\TT_\ell - \#\TT_0)^{-s} \to 0
 \quad \text{as } \ell \to \infty.
\end{align*}
One says that the Algorithm~\ref{algorithm} is rate-optimal, if and only if for all $s > 0$ with $\norm{\uu}{\mathbb{A}_s} < \infty$, the sequence of meshes $(\TT_\ell)_{\ell \in \N_0}$ generated by Algorithm~\ref{algorithm} guarantees that 
\begin{align*}
 \norm{\uu^\star - \uu_\ell^\star}{\V(\Omega)}
 \lesssim (\#\TT_\ell - \#\TT_0)^{-s} \to 0
 \quad \text{as } \ell \to \infty.
\end{align*}
We refer to~\cite{axioms} for an abstract framework of the state of the art of rate-optimal adaptive algorithms and note that LSFEM guarantees the equivalence
\begin{align}\label{eq:quasi-monotone}
 \eta_\coarse(\uu_\coarse^\star)
 \stackrel{\eqref{eq:reliable-efficient}}\simeq 
 \norm{\uu^\star - \uu_\coarse^\star}{\V(\Omega)}
 \stackrel{\eqref{eq:cea}}\simeq \min_{\vv_\coarse \in \V_\coarse(\Omega)} \norm{\uu^\star - \vv_\coarse}{\V(\Omega)}
 \stackrel{\eqref{eq:reliable-efficient}}\simeq 
 \min_{\vv_\coarse \in \V_\coarse(\Omega)} \eta_\coarse(\vv_\coarse)
\end{align}
so that the present definition of $\norm{\uu}{\mathbb{A}_s}$ is equivalent to that of~\cite{bdd2004,stevenson2007,ckns2008,axioms}. 

The following result is a direct consequence of discrete reliability~\eqref{eq:discrete_reliability} and the analysis from~\cite{axioms} under the following usual assumptions: First, suppose that we employ newest vertex bisection~\cite{stevenson:NVB,KPP13} for mesh-refinement and that the initial mesh $\TT_0$ satisfies the admissibility condition from~\cite{stevenson:NVB} for $d \ge 3$. Second, suppose that Algorithm~\ref{algorithm} employs the D\"orfler marking criterion with minimal cardinality, i.e., for given $0 < \theta \le 1$, it holds that
\begin{align}
 \MM_\ell \in \mathbb{M}_\ell := \set{\UU_\ell \subseteq \TT_\ell}{\theta \, \eta_\ell(u_\ell)^2 \le \eta_\ell(\UU_\ell, u_\ell)^2}
 \text{ \ \ and \ } \#\MM_\ell \le \#\UU_\ell 
 \text{ \ for all }\ \UU_\ell \subseteq \TT_\ell.
\end{align}

\begin{proposition}\label{prop:optimal}
There exists a constant $0 < \theta_{\rm opt} \le 1$ such that for all $0 < \theta < \theta_{\rm opt}$, the following implication holds: If there holds linear convergence~\eqref{eq:linear_convergence}, then Algorithm~\ref{algorithm} is even rate optimal, i.e., for all $s > 0$, there exists a constant $C_{\rm opt} > 0$ such that
\begin{align}\label{eq:optimal}
 C_{\rm opt}^{-1} \, \norm{\uu}{\mathbb{A}_s}
 \le \sup_{\ell \in \N_0} (\#\TT_\ell - \#\TT_0 + 1)^s \norm{\uu^\star - \uu_\ell^\star}{\V(\Omega)}
 \le C_{\rm opt} \, \norm{\uu}{\mathbb{A}_s}.
\end{align}
\end{proposition}

\begin{proof}[Sketch of proof]
According to the triangle inequality, the built-in least-squares error estimator is stable on non-refined elements in the sense of~\cite[Axiom~(A1)]{axioms}. Together with discrete reliability~\eqref{eq:discrete_reliability}, this implies optimality of the D\"orfler marking; see~\cite[Section~4.5]{axioms}. Due to~\eqref{eq:quasi-monotone}, the error estimator is quasi-monotone~\cite[Eq.~(3.8)]{axioms} with respect to mesh-refinement. This implies the so-called comparison lemma; see~\cite[Lemma~4.14]{axioms}. Together with linear convergence~\eqref{eq:linear_convergence}, \cite[Proposition~4.15]{axioms} concludes optimality~\eqref{eq:optimal}.
\end{proof}

\begin{remark}
We note that the (constrained) optimality result of Proposition~\ref{prop:optimal} can be obtained for any of the examples presented in Section~\ref{sec:examples}. For the needed $\boldsymbol{H}(\ccurl;\Omega)$-stable local projection, we refer, e.g., to~\cite[Section~4]{MR2869030}.
\end{remark}

\end{document}